\documentclass{amsart}
\usepackage{amsmath}
\usepackage{tikz}
\usepackage[T1]{fontenc}
\usepackage{mathtools}
\usepackage{url}

\newtheorem{theorem}{Theorem}[section]
\newtheorem{lemma}[theorem]{Lemma}
\newtheorem{proposition}[theorem]{Proposition}
\newtheorem{Coro}[theorem]{Corollary}

\theoremstyle{definition}
\newtheorem{definition}[theorem]{Definition}

\theoremstyle{definition}
\newtheorem{notation}[theorem]{Notation}
\newtheorem{rem}[theorem]{Remark}

\newtheorem*{theonn}{Theorem I}

\newcommand{\Rset}{\mathbb{R}}
\newcommand{\Nset}{\mathbb{N}}
\newcommand{\Zset}{\mathbb{Z}}
\newcommand{\Si}{\mathbb{S}^1}

\newcommand{\CS}{\mathfrak{F}}
\newcommand{\Ste}{{}^{st}\hspace{-0.2em}H_1}
\newcommand{\Stek}{{}^{st}\hspace{-0.2em}H}

\newcommand{\ppoint}{%
  \mathchoice
    {\vcenter{\hbox{\scalebox{0.70}{$[\;\cdot\;]$}}}} 
    {\vcenter{\hbox{\scalebox{0.70}{$[\;{\cdot}\;]$}}}} 
    {\vcenter{\hbox{\scalebox{0.70}{$\scriptstyle[\;\cdot\;]$}}}} 
    {\vcenter{\hbox{\scalebox{0.70}{$\scriptscriptstyle[\;\cdot\;]$}}}} 
}

\newcommand{\pplus}{%
  \mathchoice
    {\vcenter{\hbox{\scalebox{0.70}{$[+]$}}}} 
    {\vcenter{\hbox{\scalebox{0.70}{$[+]$}}}} 
    {\vcenter{\hbox{\scalebox{0.70}{$\scriptstyle[+]$}}}} 
    {\vcenter{\hbox{\scalebox{0.70}{$\scriptscriptstyle[+]$}}}} 
}

\newcommand{\mminus}{%
  \mathchoice
    {\vcenter{\hbox{\scalebox{0.70}{$[-]$}}}} 
    {\vcenter{\hbox{\scalebox{0.70}{$[-]$}}}} 
    {\vcenter{\hbox{\scalebox{0.70}{$\scriptstyle[-]$}}}} 
    {\vcenter{\hbox{\scalebox{0.70}{$\scriptscriptstyle[-]$}}}} 
}

\newcommand{\invsys}[1]{%
  \leavevmode
  \vtop{\offinterlineskip
    \ialign{##\cr
      $#1$\cr
      \noalign{\kern 0.1ex} 
      \hspace{-0.31em}$\leftarrow$\hidewidth\cr
    }
  }
}

\begin{document}

\title[Plane separating continua inscribe rectangles]
{Plane separating continua inscribe rectangles}

\author{Ulises Morales-Fuentes}
\address{Ulises Morales-Fuentes, Centro de Investigación en Ciencias\\
    Instituto de Investigación en Ciencias Básicas y Aplicadas\\
    Universidad Autónoma del Estado de Morelos\\
    Av. Universidad 1001\\
    Cuernavaca\\ 
    62209\\ 
    Morelos\\ 
    M\'exico.
}
\email{ulises.morales@uaem.mx, \textsc{orcid:} 0000-0003-4574-3716}

\author{Cristina Villanueva-Segovia}
\address{Cristina Villanueva-Segovia, Universidad Nacional Autónoma de México\\
        Instituto de Matemáticas\\
        Unidad Cuernavaca, Av. Universidad s/n, 
        Cuernavaca\\
        62210\\ 
        Morelos\\
        México.
}
\email{cristina@im.unam.mx, \textsc{orcid:}0009-0008-9340-0690}

\thanks{CVS acknowledges the support of the \textit{Secretaría de Ciencia, Humanidades, Tecnología e Innovación} (\textsc{secihti}) through the posdoctoral fellowship granted (Grant No. [I1200/111/2024]}

\subjclass[2020]{Primary 54Fxx; Secondary 54C56, 55N07}

\keywords{square peg problem, rectangular peg problem, inscribed rectangle, plane continua, non-locally connected continua, Steenrod homology, shape theory.}

\begin{abstract} 
We prove the following: If $X$ is a plane separating continuum, then every embedding of $X$ into $\Rset^2$ contains the vertices of a Euclidean rectangle.
We arrive to this result by extending a known result by H. Vaughan for Jordan curves to a wider class of topological objects via shape theory and Steenrod homology.
\end{abstract}

\maketitle

\section{Introduction}

At the beginning of the 20th century Toeplitz conjectured that every Jordan curve contains the four vertices of a Euclidean square. The problem of either proving or disproving Toeplitz' conjecture is now known as the square peg problem. As of today the question remains open.

Much effort has been put on the problem and many cases have been solved. It is known that if a Jordan curve satisfies one of the following conditions, then it must contain the four vertices of a square: 
piecewise analytic~\cite{Emch},
locally monotone~\cite{strom},
smooth~\cite{Shnir},
centrally symmetric~\cite{NielsenWright}, and
being the union of two Lipschitz graphs~\cite{Tao}.

This is not an exhaustive list of the solved cases of the problem but it surely serves to point out that a Jordan curve for which the square peg problem is unsolved, is difficult to define explicitly and must exhibit bad behavior.  
However, ``bad behavior'' is generic among Jordan curves; nowhere differentiable, completely non-rectifiable and having Minkowski dimension 2 are generic properties in the space of Jordan curves equipped with the supremum metric.

In contrast, the problem is completely solved for other polygons: in the negative for polygons with more than four vertices (it is an exercise); in the positive for three vertices. In~\cite{Meyerson}, Meyerson proved that given a triangle and a Jordan curve it is always possible to find three points on the curve that are the vertices of a triangle similar to the given one.

Regarding polygons with four vertices, the most general and known result is due to H. Vaughan, see~\cite[p. 71]{Vaughan}. He proved that \emph{every} Jordan curve contains four points that are the vertices of a rectangle, in other words, every Jordan curve \textsl{admits an inscribed rectangle}. He gave an elegant proof of this fact by showing that the existence of a Jordan curve that does not admit an inscribed rectangle, induces an embedding of the second symmetric product of $\Si$ into $\Rset^3$, which in turn induces an embedding of the projective plane (and of the Klein bottle) into $\Rset^3$, a contradiction. 

If we assume further that the Jordan curve is smooth, much more can be said. In 2021 Greene and
Lobb showed that every smooth Jordan curve inscribes a rectangle similar to any given rectangle~\cite{GreeneLobb}. The more general question of whether the same conclusion remains valid for every Jordan curve, is still open. This is known as the rectangular peg problem.
That same year, another generalization appeared pointing in a different direction, a topological one. 
Given a topological space $X$ that can be embedded in $\Rset^2$ we can ask whether every embedding of $X$ into the plane contains the vertices of a square, a rectangle, etc. Thus, the problem generalizes to a wider class of topological spaces. 

In~\cite{Moravilla} it is shown that the arc and the simple triod are the only locally connected plane continua for which an embedding into the plane exists that does not contain the vertices of a rectangle.
From there, a subsequent question arises: To what extent is the local behavior involved in these ``inscription properties'' of a topological space? For instance, looking at Vaughan's result, the fact that every embedding of $\Si$ into the plane admits an inscribed rectangle, does it rely on the local connectedness of $\Si$ or is it solely a consequence of the ``general shape'' of $\Si$?
Based on~\cite{Moravilla} and the positive cases of the square peg problem in which, to some degree, good local behavior is assumed, the roll of local connectedness appears fundamental.
Nonetheless, some results in~\cite{circhain} may represent a challenge to that supposition: the images of a particular class of embeddings of some non-locally connected plane continua into the plane admit an inscribed square.

One way in which we can ``isolate'' the general behavior of $\Si$ from the local behavior is recalling that a subset $\Gamma$ of the plane is a Jordan curve if and only if: 
\begin{enumerate}
    \item $\Rset^2\setminus \Gamma$ has exactly two components.
    \item Every point of $\Gamma$ is accessible from any point of $\Rset^2\setminus \Gamma$.
\end{enumerate}
In this context the question is whether the accessibility condition can be relaxed.
It turns out ---somehow surprisingly--- that the accessibility condition is not needed at all and, even more, the first condition can be relaxed: 

\begin{theonn}[Main result]
    If $X$ is a plane separating continuum, then every embedding of $X$ into $\Rset^2$ contains the vertices of a rectangle.
\end{theonn}

This result is achieved by extending Vaughan's arguments to a wider class of topological objects via shape theory and Steenrod homology.

We divide this paper in three sections, the first one is devoted to stablish the notation and to prove some results that will be useful. In particular we introduce the 2-symmetric clamshell of a continuum. In Section~\ref{sec:main} we prove that the shape of the 2-symmetric clamshell of a plane separating continuum $X$ is completely determined by $n(X)$, the number of components of $\Rset^2\setminus X$. Then, we proceed to the proof of Theorem I, the idea is the following: using shape invariance of the Steenrod homology groups of metric compacta, we prove that the 2-symmetric clamshell of a plane separating continuum $X$ cannot be embedded in $\Rset^3$, then we use a modified version of the test function for rectangles used by H. Vaughan to reach a contradiction under the assumption that $X$ does not inscribe rectangles.
We divide the proof in cases depending on whether $n(X)$ is finite or infinite. However, we decided to give a separate proof for the case $n(X)=2$ since the homology groups needed for that case are very well known (these are the Klein bottle homology groups) and most importantly because it is a neat generalization of Vaughan's proof. We finish this paper with some conclusions and further questions. 

Since the square peg problem has been approached using tools from many different areas of mathematics, we tried to make this paper as accessible as possible by providing various references.  

\section{Preliminaries}

In this section we introduce some notation and we state and prove some lemmas that will be used in Section~\ref{sec:main}. We start with a brief description of the inscription property for continua that was introduced in \cite{Moravilla}. We then introduce the 2-symmetric clamshell of continua, this notion will be widely used throughout this work. In this section we show that the 2-symmetric clamshell is a continuous functor on the class of continua. We finish this section with the calculation of some homology groups that will be needed for the proof of the main results.  

\subsection{Inscription property for plane continua}

A \emph{continuum} is a nonempty compact, 
connected, metric space. A \emph{plane continuum} is a continuum that can be embedded in $\mathbb{R}^2$. 

Given a continuum, $X$, the \emph{second symmetric product} of $X$, denoted by $F_2(X)$, is the set $\{ A\subset X : A \neq \emptyset, |A|\leq 2 \}$ endowed with the topology induced by the Hausdorff metric.

Given a polygon $P\subseteq \Rset^2$, it is said that a subset $X$ of the plane \emph{admits an inscribed polygon} $P$ if there is a similar copy of $P$ such that all of its vertices lie in $X$. Following~\cite{Moravilla}, we will be using a topological version of this notion for plane continua:

\begin{definition}
    Given a polygon $P\subseteq \Rset^2$. We say that a plane continuum $X$ inscribes $P$ if for every embedding $\gamma:X\to\Rset^2$ we have that $\gamma(X)$ admits an inscribed polygon $P$. In particular, we say that $X$ inscribes rectangles if the image of every embedding of $X$ into the plane contains four points that are the vertices of a Euclidean rectangle.
\end{definition}

It is clear that under this definition the inscription property is a topological property for plane continua. It is worth pointing out here that with this terminology the square peg problem asks whether $\Si$ inscribes squares. On the other hand, it is already known that $\Si$ inscribes any type of triangles~\cite{Meyerson} and, as it has been mentioned in the introduction, H. Vaughan proved that $\Si$ inscribes rectangles (see ~\cite{Vaughan}).

\subsection{The 2-symmetric clamshell of continua}

\begin{definition}
    Given a compact metric space $X$, \textsl{the 2-symmetric clamshell of} $X$, denoted by $\CS(X)$, will be the adjunction space formed by attaching two copies of the second symmetric product of $X$ along its subspace of singletons. 
\end{definition}

\begin{notation}\label{CSNotation}
    In what follows, we will be using plus and minus signs to distinguish the two foundational copies of $F_2(X)$ that define the 2-symmetric clamshell of $X$. To be precise, let $i_1$ denote the identity map on $F_2(X)$ restricted to the closed subset $S=\{\{a,b\}\in F_2(X):a=b\}$, then:
    $$\CS(X)=F_2(X)^{\pplus}\cup_{i_1} F_2(X)^{\mminus}
    =(F_2(X)^{\pplus}\sqcup F_2(X)^{\mminus})/\sim,$$
    where $x\sim y$ iff $i_1(y)=x$. We will be referring to images under the quotient map $q:F_2(X)^{\pplus}\sqcup F_2(X)^{{\mminus}}\to \CS(X)$ in the following way:
    \[
        q(F_2(X)^{\pplus})=\CS(X)^{\pplus}, \hspace{0.5em}
        q(F_2(X)^{\mminus})=\CS(X)^{\mminus}, \hspace{0.5em}
        q(S^{\pplus}\sqcup S^{\mminus}) = X^\star.
    \]
    Furthermore, we will also use plus and minus sings to indicate if a point of $\CS(X)$ belongs to 
    $\CS(X)^{ \pplus}$ or to $\CS(X)^{\mminus}$, as follows: a point $x\in\CS(X)$ will be denoted by $\{a,b\}^{\ppoint}$ with $\ppoint\in\{\pplus,\mminus\}$ iff $x\in\CS(X)^{\ppoint}$ and the inverse image of $x$ under the inclusion map from $F_2(X)$ to $\CS(X)^{\ppoint}$ is $\{a,b\}$. Clearly, 
    $$x=\{a,b\}^{\pplus}=\{a,b\}^{\mminus}\iff x\in X^\star\iff a=b.$$
    
    Finally, if $Y$ is a compact metric space, then any given map $f:X\to Y$ naturally induces a map between the 2-symmetric clamshells of $X$ and $Y$, \textsl{the induced map} $f^*:\mathfrak{F}(X)\to\mathfrak{F}(Y)$, defined as:
    \begin{equation*}
         f^*(\{a,b\}^{\ppoint}) =
         \begin{cases}
             \{f(a),f(b)\}^{\pplus} & \text{ if } \ppoint=\pplus \\
             \{f(a),f(b)\}^{\mminus}& \text{ if } \ppoint=\mminus
         \end{cases}  
     \end{equation*}
\end{notation}

It is routine to verify that $f^*$ is indeed continuous whenever $f$ is continuous and that $(g\circ f)^*=g^*\circ f^*$.

We will also be using a test map for rectangles that will be defined on the 2-symmetric clamshell of a continuum $X$. This map is just an extended version of the function used in the original proof for Jordan curves given by Vaughan in \cite{Vaughan}:

\begin{definition}
    For a plane continuum $X$, given an embedding $\gamma:X\hookrightarrow\Rset^2$, the \textsl{Vaughan's function} associated to $\gamma$, $\mathcal{V}_\gamma:\CS(X)\to\Rset^3$, is defined as follows:
        \[
        \mathcal{V}_{\gamma}(\{a,b\}^{\ppoint})=\begin{cases}
        \left(\frac{\gamma(a)+\gamma(b)}{2},
        \|\gamma(a)-\gamma(b)\|\right), 
        & \mbox{if}\ a\neq b \text{ and }  \ppoint=\pplus;\\
        \left(\frac{\gamma(a)+\gamma(b)}{2},
        -\|\gamma(a)-\gamma(b)\|\right), 
        & \mbox{if}\ a\neq b \text{ and } \ppoint=\mminus;\\
        \left(\gamma(a),0\right), 
        & \mbox{if}\ a=b.
        \end{cases}
        \]
\end{definition}

\begin{rem}\label{Vau}
Observe that $\mathcal{V}_{\gamma}$ is continuous and that
its restriction to the singletons is an embedding of $X$ into the plane $z=0$. Furthermore, $\mathcal{V}_{\gamma}$
is one to one if and only if $\gamma(X)$
does not admit an inscribed rectangle.
\end{rem}

\begin{lemma}\label{CShom}
    Let $X$, $Y$ be two continua. If $f:X\to Y$ is homotopic to $g:X\to Y$, then $f^*:\CS(X)\to\CS(Y)$ is homotopic to $g^*:\CS(X)\to\CS(Y)$.
\end{lemma}

\begin{proof}
    Let $f$ and $g$ as in the statement and let $h$ be an homotopy between $f$ and $g$ in $Y$. Define $H:\CS(X)\times [0,1]\to\CS(Y)$ as 
    $$H(\{a,b\}^{\ppoint}, t)=\{h(a,t),h(b,t)\}^{\ppoint}$$
    Note that for a fixed $t\in[0,1]$, $H_t(\{a,b\}^{\ppoint})=h^*_t(\{a,b\}^{\ppoint})$ which is continuous. This last equation also gives $H_0=f^*$ and $H_1=g^*$. Also, since $\{(a,t):h_t(a)\in U\}$ is open in $X\times [0,1]$ for any open set $U$ of $Y$, we can easily show that for any open set $U_*$ of $\CS(Y)$, the set $\{(\{a,b\}^{\ppoint},t):H_t(\{a,b\}^{\ppoint})\in U_*\}$  is open in $\CS(X)\times [0,1]$.
\end{proof}

\begin{lemma}\label{invlim}
    Let $X$ be a continuum, if $X=\varprojlim (X_\alpha, f_{\alpha\alpha'},\Omega)$, then
    
    $$\CS(X)\cong\varprojlim (\CS(X_\alpha),f^*_{\alpha \alpha'},\Omega)$$.
\end{lemma}

\begin{proof}

    For $\alpha \in \Omega$, let $\pi_\alpha$ denote the projection of $X$ to $X_\alpha$ and consider the induced maps $\pi_\alpha^*:\CS(X)\to\CS(X_\alpha)$.
    Also, let $Y$ denote the inverse limit $(\CS(X_\alpha),f^*_{\alpha \alpha'},\Omega)$ and let $\rho_\alpha:Y\to\CS(X_\alpha)$ denote the projection map 
    of $Y$ to $\CS(X_\alpha)$. 
    
    It is easy to check that for $\alpha\leq\alpha'$, we have $f^*_{\alpha\alpha'}\circ\pi^*_{\alpha'}=\pi^*_\alpha$. Therefore, there exists a unique map 
    \begin{equation}\label{eq:liminv}
          h:\CS(X)\to Y \text{ such that }\pi^*_\alpha=\rho_\alpha\circ h.
    \end{equation}
    
    As we shall see the map $h$ is in did a homeomorphism.
    We first show that $h$ is one to one. Say $h(\hat{x}_1)=h(\hat{x}_2)$, for some $\hat{x}_1=\{a_1,b_1\}^i$, $\hat{x}_2=\{a_2,b_2\}^j$, with $i,j\in\{\pplus, \mminus\}$. 
    
    From~\eqref{eq:liminv} it follows that for all $\alpha\in\Omega$, we have 
    $\pi^*_\alpha(\hat{x}_1)=\pi^*_\alpha(\hat{x}_2)$. By definition of $\pi_\alpha^* $, this means that     
    \begin{equation}\label{eq:liminv2}
        \{\pi_\alpha(a_1),\pi_\alpha(b_1)\}^j=\{\pi_\alpha(a_2),\pi_\alpha(b_2)\}^i, \forall \alpha\in\Omega
    \end{equation}
    
    Thus, we get $i=j$ and also it is clear that for $a_1=b_1$, we have 
    $$\pi_\alpha(a_1)=\pi_\alpha(b_1)=\pi_\alpha(a_2)=\pi_\alpha(b_2),\; \forall \alpha\in\Omega,$$
    
    and we get $\hat{x}_1=\hat{x}_2$. We assume now that $a_1\neq b_1$. Fix $\beta\in\Omega$, we may assume further that $\pi_{\beta}(a_1)= \pi_{\beta}(a_2)$, it follows that $\pi_{\alpha}(a_1)= \pi_{\alpha}(a_2)$ for all $\alpha\in\Omega$, otherwise there is some $\alpha\in\Omega$ such that $\pi_\alpha(a_1)=\pi_\alpha(b_2)$. Since $a,b\in X$, the commutativity between the projections and the transition maps gives:
    \begin{align*}
        \text{If }& \beta<\alpha, \text{ then } \pi_{\beta}(a_1)=f_{\beta\alpha}(\pi_\alpha(a_1))=f_{\beta\alpha}(\pi_\alpha(b_2))=\pi_{\beta}(b_2),\\
        \text{if }& \beta>\alpha, \text{ then } \pi_{\alpha}(a_1)=f_{\alpha\beta}(\pi_\beta(a_2))=f_{\alpha\beta}(\pi_\beta(a_2))=\pi_{\alpha}(a_2),
    \end{align*}
   
    a contradiction in any case. Therefore $\pi_{\alpha}(a_1)\neq \pi_{\alpha}(b_2)$
    for all $\alpha\in\Omega$. This observation, together with~\eqref{eq:liminv2} implies that $\pi_\alpha(a_1)=\pi_\alpha(a_2)$ and $\pi_\alpha(b_1)=\pi_\alpha(b_2)$ for all $\alpha\in\Omega$. We conclude that $\hat{x}_1=\hat{x}_2$ and that $h$ is one to one.

    We now show that $h$ is surjective. Take $(y_\alpha)_{\alpha\in\Omega}\in Y$, so $y_\alpha=\{a_\alpha,b_\alpha\}^{i_\alpha}$ for some $\{a_\alpha,b_\alpha\}^{i_\alpha}\in \CS(X_\alpha)$, $i_\alpha\in\{\pplus,\mminus\}$ such that $f^*_{\alpha\beta}(y_\beta)=y_\alpha$ for all $\alpha\leq\beta$, by definition of the induced map, we get:
    $$\{f_{\alpha\beta}(a_\beta),f_{\alpha\beta}(b_\beta)\}^{i_\beta}=\{a_{\alpha},b_{\alpha}\}^{i_\alpha}\; \forall\alpha,\beta\in\Omega, \alpha\leq\beta$$
    This equation readily shows that $i_\alpha=i_\beta$ for all $\alpha,\beta\in\Omega$, say $i_\alpha=\ppoint$ for all $\alpha\in\Omega$.
    
    Now, for each $\alpha\in\Omega$, pick some $\alpha'>\alpha$ and denote by $x_{a,\alpha}$ an element in $\{a_\alpha,b_\alpha\}$ that is equal to $f_{\alpha{\alpha'}}(a_{\alpha'})$. Note that this definition does not depend on the choice of $\alpha'$. In the same way we define $x_{b,\alpha}$ to be an element of $\{a_\alpha,b_\alpha\}$ equal to $f_{\alpha\alpha'}(b_{\alpha'})$.
    It is clear that for $\hat{x}_a=(x_{a,\alpha})_{\alpha\in\omega}$ and $\hat{x}_b=(x_{b,\alpha})_{\alpha\in\Omega}$ we have $f_{\alpha{\alpha'}}(x_{k,\alpha'})=x_{k,\alpha}$ for $k\in\{a,b\}$, $\alpha\in\Omega$. Thus $x_{k,\alpha}\in X$ for $k\in\{a,b\}$ and $\{\hat{x}_a,\hat{x}_b\}^{\ppoint}\in\CS(X)$.
    Finally, to conclude that $h$ is a homeomorphism, simply notice that:
    \begin{align*}
    \rho_\alpha\circ h(\{\hat{x}_a,\hat{x}_b\}^{\ppoint})&= \pi_\alpha^*(\{\hat{x}_a,\hat{x}_b\}^{\ppoint})=
    \{\pi_\alpha(\hat{x}_a),\pi_\alpha(\hat{x}_b)\}^{\ppoint}
    \\
    &=\{x_{a,\alpha},x_{b,\alpha}\}^{\ppoint}=\{a_\alpha,b_\alpha\}^{i_\alpha}.
    \end{align*}
    Thus $h(\{\hat{x}_a,\hat{x}_b\}^{\ppoint})=y$ and we finish the proof. 
\end{proof}

\subsection{Homology of the wedge of circles and of its second symmetric product}

All homology groups throughout this work are taken with coefficients in $\Zset$, so we will only write $H_n(X)$ to denote the $n$-th singular homology group with integer coefficients of a topological space $X$. We use $\Stek_n(X)$ for the Steenrod homology groups of $X$. We use $\check{H}^n(X)$ and $H^n(X)$ to denote the $n$-th \v{C}hech and singular cohomology groups of $X$, respectively.

In \cite{Tuffley}, it is studied the homology of a quotient space of $X^n$ defined via the map $ (x_1, \dots, x_k) \mapsto \{x_1\} \cup \dots \cup \{x_k\}$. In \cite{Castillo} the homology of the second symmetric product (endowed with the topology induced by the Hausdorff metric) is calculated. Note that in general, $X^n/S_n$ where $S_n$ is the group of permutations of the coordinates of $X^n$ is not homeomorphic to $F_n(X)$ (see \cite[Introduction]{chinen2010symmetric}); nonetheless, if $n=2$ they coincide. In our context, we have that the following lemma is a particular case of \cite[Theorem 1]{Tuffley}; and it is also a particular case of \cite[Theorem 2.4.9]{Castillo}:

\begin{lemma}\label{Whomos}
    The first and second singular homology groups of the second symmetric product of the wedge of $k$ circles, $Y_k=\vee_{i=1}^k \Si$, are:
    $$H_1(F_2(Y_k))\cong \Zset^{k},\text{ and } H_2(F_2(Y_k))\cong \Zset^{\frac{1}{2}k(k-1)}.$$
\end{lemma}

This last lemma enables us to compute the Steenrod homology groups of the Hawaiian earing and of its second symmetric product using continuity of Steenrod.

\begin{lemma}\label{HawaiiHomolo}
    Let $\mathcal{H}=\varprojlim Y_k$, where $Y_k$ is the wedge of $k$ circles. The Steenrod homology groups of the Hawaiian earing $\mathcal{H}$ and of its second symmetric product $F_2(\mathcal{H})$ satisfy:
        $$\Ste(\mathcal{H})\cong\Ste(F_2(\mathcal{H}))\cong \Zset^\omega$$
\end{lemma}

\begin{proof}
     First we prove that $\Ste(\mathcal{H})\cong \Zset^\omega$. This can be derived from Milnor's continuity theorem~\cite{milnor}, as it gives the exact sequence:

    $$0 \to \varprojlim\nolimits^1 H_2(Y_k) \to \Ste(\mathcal{H}) \to \varprojlim H_1(Y_k) \to 0, $$
    where $\varprojlim^1$ denotes the first derived limit functor. However, since $H_2(Y_k)= 0$ for all $k\in\Nset$ and $H_1(Y_k)=\Zset^n$, the sequence simplifies to:
    $$0\to \Ste(\mathcal{H})\to \varprojlim \Zset^k\to 0$$
    So we get $\Ste(\mathcal{H})\cong \varprojlim \Zset^k=\Zset^\omega$.

    We now deal with the first homology group of the second symmetric product of $\mathcal{H}$. This can be calculated using again the continuity theorem and Lemma~\ref{Whomos} as follows: consider the exact sequence
    \begin{equation}\label{eq:Hlim}
        0 \to \varprojlim\nolimits^1 H_2(F_2(Y_k)) \to \Ste(F_2(\mathcal{H})) \to \varprojlim H_1(F_2(Y_k)) \to 0
    \end{equation}
      
    This time $H_2(F_2(Y_k))\cong \Zset^{\frac{1}{2}k(k-1)}$ (from Lemma~\ref{Whomos}), so to see that the first derived limit vanishes we verify that the bonding maps of the inverse system $(\Zset^{\frac{1}{2}k(k-1)},p_{k,k+1},\Nset)$ are surjective. Note that the bonding map $p_{k,k+1}:\Zset^{\frac{1}{2}(k+1)k}\to\Zset^{\frac{1}{2}k(k-1)}$ is induced by the map $P_{k,k+1}:F_2(Y_{k+1})\to\ F_2(Y_{k})$ that is the identity in a subspace homeomorphic to $F_2(Y_{k})$ and collapses its complement to a single point (the complement being the union of $k-1$ tori and a M\"obius strip). Hence, $p_{k,k+1}$ must be the identity on a subgroup isomorphic to $\Zset^{\frac{1}{2}k(k-1)}$ and thus $p_{k,k+1}$ must be surjective.

    Therefore, by the Mittag-Leffler condition, the sequence~\eqref{eq:Hlim} simplifies to:
          $$0 \to \Ste(F_2(\mathcal{H})) \to \varprojlim \Zset^k \to 0.$$
    and exactness gives: $\Ste(F_2(\mathcal{H}))\cong \varprojlim \Zset^k=\Zset^\omega$.
\end{proof}

\begin{rem}\label{rem:onhomo}
    In order to calculate the homology groups of the 2-symmetric clamshell of a plane continuum we will be using shape invariance and the duality between the Steenrod-Sitnikov homology and the \v{C}ech cohomology. It is known that, in general, the Steenrod-Sitnikov homology is not a shape invariant, however strong homology is a strong shape invariant~\cite[Theorem 19.1]{StrongMardesic} and, moreover, for metric compacta the strong homology coincides with the Steenrod-Sitnikov homology, since it satisfies the Eilenberg-Steenrod axioms~\cite[Corollary 19.31]{StrongMardesic}. Additionally, strong shape is shape invariant for compact metric spaces~\cite[Theorem 9.19]{StrongMardesic}. 
    Thus, for the class of metric compacta the Steenrod-Sitnikov homology is a shape invariant.
\end{rem}

\section{Main Results}\label{sec:main}

This section is divided in four parts, the first, is devoted to prove that the shape of the 2-symmetric clamshell of a plane continuum is determined entirely by the number of components of $\Rset^2\setminus X$. In the second part we show that all continua that separate the plane in exactly two components inscribe rectangles. As mentioned in the introduction, we decided to keep a separate place for this case because it provides a neat generalization of Vaughan's proof. The last two parts of this section deal with the cases of finite and infinite number of components, respectively.

\subsection{The shape of the 2-symmetric clamshell of plane continua}

In~\cite[Theorem~3]{Kodama} and~\cite[Theorem 2.4]{Oledzki} it has been proven that the shapes of the second symmetric products of two continua are the same, provided the base spaces have the same shape. In the next lemma, we prove that the 2-symmetric clamshell of continua has the same property. Since we will be working only with plane continua as base spaces, by Borsuk's Theorem~\cite[Theorem~9.1]{Borsuk},
we can give our statement in terms of the number of components in which the base spaces separate the plane rather than in terms of their shape.
We refer the reader to~\cite{MardesicC} for the definitions and notation that we will be using.

\begin{lemma}\label{sameshape}
    Given a plane continuum $W$, the shape of the 2-symmetric clamshell of $W$ is completely determined by the number of components of the complement of $W$ in the plane.
\end{lemma}

\begin{proof} 
    Let $X$ and $Y$ be two plane continua that separate the plane in exactly the same number of components. 
    Since $X$ and $Y$ are compact metric spaces, we know from~\cite[Theorem~7]{MardesicC}, that there are  inverse systems, say 
    ${\underline{X}}= (X_\alpha, \varphi_{\alpha\alpha'},\Omega)$
    and $\underline{Y}=(Y_\beta, \gamma_{\beta\beta'},\Gamma)$
    such that for each $\alpha\in\Omega$, $\beta\in\Gamma$, $X_\alpha$, and $Y_\beta$ are absolute neighborhood retracts (\textsc{anr}'s) and, $$X=\varprojlim(X_\alpha, \varphi_{\alpha\alpha'},\Omega) \:\text{ and }\:Y=\varprojlim(Y_\beta, \gamma_{\beta\beta'},\Gamma).$$
    
    Since $X$ and $Y$ are plane continua that separate the plane in the same number of components, by Borsuk's Theorem (\cite[Theorem~9.1]{Borsuk}) $X$ and $Y$ have the same shape type.
    Therefore, the inverse systems $\underline{X}$ and $\underline{Y}$ are of the same homotopy type as defined in~\cite[Section~4]{MardesicC}. This means that there exist maps of systems 
    $\underline{F}: \underline{X}\rightarrow \underline{Y}$ and $\underline{G}: \underline{Y} \rightarrow \underline{X}$ such that $\underline{G}\circ\underline{F} \simeq \underline{1}_{\underline{X}}$ and $\underline{F}\circ\underline{G}\simeq \underline{1}_{\underline{Y}}$.
    Say $\underline{F}=(\{f_\beta\}_{\beta\in\Gamma},f)$ and $\underline{G}=(\{g_\alpha\}_{\alpha\in\Omega},g)$, where $f:\Gamma\to\Omega$ and $g:\Omega\to\Gamma$ are increasing functions and $f_\beta:X_{f(\beta)}\to Y_\beta$
    $g_\alpha:Y_{g(\alpha)}\to X_\alpha$ are the collections of maps such that for all $\alpha,\alpha'\in\Omega$, $\beta,\beta'\in\Gamma$ with $\alpha\leq\alpha'$, $\beta\leq\beta'$, we have:
    \begin{equation}\label{eq1:sameshape}
        \begin{aligned}
            f_\beta\circ\varphi_{f(\beta)f(\beta')}&\simeq \gamma_{\beta\beta'}\circ f_{\beta'}\\
            g_\alpha\circ\gamma_{g(\alpha)g(\alpha')}&\simeq \varphi_{\alpha\alpha'}\circ g_{\alpha'}
        \end{aligned}
    \end{equation}
    
    Using the induced maps on the clamshells of $X$ and $Y$, we can now consider the inverse systems given by  
    $\underline{\CS(X)}=(\CS(X_\alpha),\varphi^*_{\alpha,\alpha'},\Omega)$
    and 
    $\underline{\CS(Y)}=(\CS(Y_\beta), \gamma^*_{\beta,\beta'},\Gamma)$. 
    
    Note that the 2-symmetric clamshell of an \textsc{anr}, is an \textsc{anr}. This follows from the fact that the second symmetric product of an \textsc{anr} is an \textsc{anr}, (see~\cite{Gaena} and~\cite{Jawo}) and Borusk's theorem on the union of \textsc{anr'}s~\cite[Theorem 6.1, (ii)]{BorsukB}.
    Therefore, $\underline{\CS(X)}$ and $\underline{\CS(Y)}$  are indeed inverse systems of \textsc{anr}'s. Moreover, from Lemma~\ref{invlim} we know that these inverse systems are associated with the clamshells of $X$ and $Y$ since: 
    $$\CS(X)\cong \varprojlim(\CS(X_\alpha),\varphi^*_{\alpha,\alpha'},\Omega)\:\text{ and }\:\CS(Y)\cong\varprojlim(\CS(Y_\beta), \gamma^*_{\beta,\beta'},\Gamma).$$
    
    We now define $\underline{F}^*=(\{f^*_\beta\}_{\beta\in\Gamma},f)$ and 
    $\underline{G}^*=(\{g^*_\alpha\}_{\alpha\in\Gamma},g)$. As we shall see these are maps of systems that are homotopy equivalences.
    Actually, from Lemma~\ref{CShom} and the fact that the induced map of a composition is the composition of the induced maps
    we can derive directly from~\eqref{eq1:sameshape} that for all $\alpha,\alpha'\in\Omega$, $\beta,\beta'\in\Gamma$ with $\alpha\leq\alpha'$, $\beta\leq\beta'$, we have:
    \begin{align*}
        f^*_\beta\circ\varphi^*_{f(\beta)f(\beta')}&\simeq \gamma^*_{\beta\beta'}\circ f^*_{\beta'}\\
        g^*_\alpha\circ\gamma^*_{g(\alpha)g(\alpha')}&\simeq \varphi^*_{\alpha\alpha'}\circ g^*_{\alpha'}
    \end{align*}
    Hence, $\underline{F}^*$ and $\underline{G}^*$ are maps of systems. It remains to show that
    $\underline{G}^*\circ\underline{F}^* \simeq \underline{1}_{\underline{\CS(X)}}$ and $\underline{F}^*\circ\underline{G}^*\simeq \underline{1}_{\underline{\CS(Y)}}$. Since $\underline{F}:\underline{X}\to\underline{Y}$ is a homotopy equivalence we know that for all $\alpha\in\Omega$, $\beta\in\Gamma$ there is $\alpha'\in\Omega$ and $\beta'\in\Omega$ with $\alpha'\geq \alpha$ and $\beta'\geq\beta$ such that:
    \begin{equation}\label{eq2:sameshape}
        \begin{aligned}
            g_\alpha\circ f_{g(\alpha)}\circ\varphi_{(f\circ g(\alpha))\alpha'}&\simeq 1_{X_\alpha}\circ \varphi_{\alpha\alpha'}\\
            f_\beta\circ g_{f(\beta)}\circ\gamma_{(g\circ f(\beta))\beta'}&\simeq 1_{Y_\beta}\circ \gamma_{\beta\beta'}
        \end{aligned}
    \end{equation}
    As before, from Lemma~\ref{CShom} we get: 
    
    \begin{equation}
        \begin{aligned}
            g^*_\alpha\circ f^*_{g(\alpha)}\circ\varphi^*_{(f\circ g(\alpha))\alpha'}&\simeq 1_{\CS(X_\alpha)}\circ \varphi^*_{\alpha\alpha'}\\
            f^*_\beta\circ g^*_{f(\beta)}\circ\gamma^*_{(g\circ f(\beta))\beta'}&\simeq 1_{\CS(Y_\beta)}\circ \gamma^*_{\beta\beta'}
        \end{aligned}
    \end{equation}
    
    We conclude that $\underline{F}^*$ and $\underline{G}^*$ induce a homotopy equivalence between the inverse systems associated with $\CS(X)$ and $\CS(Y)$, thus $\CS(X)$ and $\CS(Y)$ have the same shape type.
\end{proof}

\subsection{The case of continua separating the plane into exactly two components}

\begin{proposition}\label{CSNoEmb2compos}
    Let $W$ be a plane separating continuum, such that $\mathbb{R}^2 \setminus W$ has exactly two components, then the 2-symmetric clamshell of $W$ is not embeddable in $\Rset^3$.
\end{proposition} 

\begin{proof}
    Let $W\subseteq \mathbb{R}^2$ as in the statement, by Lemma~\ref{sameshape},  $\CS(W)$ has the same shape type as $\CS(\Si)$. It is well known that $\CS(\Si)$ is the Klein bottle. 
    This implies (see Remark~\ref{rem:onhomo}) that the Steenrod homology groups of $\CS(W)$ are isomorphic to the homology groups of the Klein bottle, $K$. 
    
    Now, assume that $\CS(W)$ can be embedded in $\Rset^3$, note that then $\mathbb{S}^3 \setminus \CS(W)$ is a manifold. Using the Steenrod-Sitnikov duality~\cite{Inassa, Steenrod} we get the following: 
    $$\Zset\oplus \Zset_2 \cong 
    H_1(K) \cong \Ste(\CS(W))
    \cong\check{H}^1(\mathbb{S}^3 \setminus \CS(W))\cong
    H^1 (\mathbb{S}^3 \setminus \CS(W)),$$

    a contradiction, since $H^1(\mathbb{S}^3 \setminus \CS(W)))$ must be torsion-free.
    
    We conclude that $\CS(W)$ is not embeddable in $\mathbb{R}^3$.
\end{proof}

\begin{theorem}\label{2compos}
    Every plane continuum that separates the plane into exactly two components admits an inscribed rectangle.
\end{theorem} 

\begin{proof}
    Let $X$ as in the statement and assume that $X$ does not admit an inscribed rectangle. Hence, there exists an embedding $\gamma:X\hookrightarrow \Rset^2$ such that $\gamma(X)$ does not admit an inscribed rectangle.  Now, from Remark~\ref{Vau}, the Vaughan's function associated to $\gamma$, $\mathcal{V}_\gamma:\CS(X)\to\Rset^3$, is an embedding of $\CS(X)$ into $\Rset^3$. This contradicts Proposition~\ref{CSNoEmb2compos}. Therefore $X$ admits an inscribed rectangle.
\end{proof}

Notably, this last Theorem implies the following:

\begin{Coro}  
    The Warsaw circle, the circle of pseudo-arcs and the pseudo-circle each inscribes rectangles.
\end{Coro}

It is worth pointing out here that that there are plane continua that separate the plane in $n$ components and do not contain any subcontinuum that separates the plane in $m$ components for any $m<n$. This is the case of the Wada continua.
Given $n\in\omega + 1$, we say that $W$ is a Wada continuum for $n$ lakes if $\Rset\setminus W$ has $n$ components, say $U_i$, $i\in\{1,\ldots n\}$, and $\partial U_i = W$ for all $i\in\{1,\ldots n\}$. These type of continua were constructed in~\cite{YoneyamaWada} and \cite{Brouwer2lakes}.  As we shall see in the next section, all Wada continua for $n$ lakes, with $n\in\omega$, inscribe rectangles. Finally, the results of the last section will include the Wada continua for $\omega$ lakes.

\subsection{Continua that separate the plane into a finite number of components.}

In this section we will prove that plane continua that separate the plane into any finite number of components admit an inscribed rectangle.

\begin{lemma}\label{HomCSnWedge}
    Let $Y=\vee_{i=1}^n \Si$ be the wedge sum of $n$ circles, then 
    $$H_1(\CS(Y))\cong \Zset^n\oplus (\Zset_2)^n $$
\end{lemma}

\begin{proof}
     Recalling the notation given in~\ref{CSNotation}, we have the following decomposition of $\CS(Y)$:
     $$\CS(Y)=\CS(Y)^{\pplus}\cup\CS(Y)^{\mminus}\text{ and }
     \CS(Y)^{\pplus}\cap\CS(Y)^{\mminus}=Y^\star.$$
     Note that this decomposition satisfies that there are open subsets $U,V$ of $\CS(Y)$ such that $\CS(Y)^{\pplus}\subseteq U$, $\CS(Y)^{\mminus}\subseteq V$ and such that $Y^\star$ is a deformation retract of $U\cap V$. Hence, in order to calculate $H_1(\CS(Y))$, we may consider the reduced Mayer-Vietoris long exact sequence for $(\CS(Y),\CS(Y)^{\pplus},\CS(Y)^{\mminus})$:
    \begin{align}\label{MVW}
        \ldots\to
         H_1(Y^\star)&\xrightarrow{\varphi} 
         H_1(\CS(Y)^{\pplus})\oplus H_1(\CS(Y)^{\mminus})\\ \notag
         &\xrightarrow{\psi}
         H_1(\CS(Y))\to 
         H_0(Y^\star)\to\ldots,
    \end{align}

    where $\varphi=(i^{+}_*,-i^{-}_*)$ and $i^{+}_*, i^{-}_*$ are the homomorphism induced by the inclusion maps $i^{\ppoint}:Y^\star\to\CS(Y)^{\ppoint}.$
    
    Since $Y^\star\cong \vee_{i=1}^n\Si$, and we are using the reduced homology, we have that $H_0(Y^\star)=0$, it is also clear that $H_1(Y^\star) \cong \Zset^n$. 
    The first observation implies that $\psi$ is a surjection. 
    Now, to determine $\varphi$, we use the second observation: 
    notice that for each copy of $\Si$ in $Y^\star$, its image under the inclusion maps $i^{\ppoint}$ is the the boundary of a M\"obius strip, and this map is homotopic to the map that winds twice around the core of the M\"obius strip. Therefore, for $\hat{z}\in\Zset^n$, we have 
    $\varphi (\hat{z})=(2\hat{z},-2\hat{z})$. For $a\in\Zset$ let us denote by $\hat{a}$ the element $(a,\ldots, a)\in\Zset^n$. From the Mayer-Vietoris sequence~\eqref{MVW} and the properties of $\varphi$ and $\psi$ we have just mentioned, we deduce that:

    $$H_1(\CS(Y))\cong \frac{H_1(\CS(Y)^{\pplus})\oplus H_1(\CS(Y)^{\mminus})}{\langle (\hat{2},-\hat{2})\rangle}.$$

    Now, from Lemma~\ref{Whomos}, it follows that $H_1(\CS(Y)^{\ppoint})\cong H_1(Y)\cong\mathbb{Z}^n$. We conclude that:
    $$H_1(\CS(Y))\cong \frac{\Zset^n\oplus\Zset^n}{\langle(\hat{2},-\hat{2}) \rangle}=\Zset^n\oplus(\Zset_2)^n.$$
\end{proof}

\begin{proposition}\label{CSNoEmbNcompos}
    Let $W$ be a plane separating continuum such that $\mathbb{R}^2 \setminus W$ has a finite number of components, then $\CS(W)$ cannot be embedded in $\Rset^3$.
\end{proposition}

\begin{proof}
    Let $W$ as in the statement, say $\Rset^2\setminus W$ has exactly $n$ components. Let $Y$ be the wedge of $n$ circles, $Y=\vee_{i=1}^n\Si$. From Lemma~\ref{sameshape}, we know that $\CS(W)$ and $\CS(Y)$ have the same shape type, it follows that $\CS(W)$ and $\CS(Y)$ have the same Steenrod homology groups (see Remark~\ref{rem:onhomo}). It follows from Lemma~\ref{HomCSnWedge} that $H_1(\CS(W))=\mathbb{Z}^n\oplus(\mathbb{Z}_2)^n$.

    Assume now that $\CS(W)$ can be embedded in $\Rset^3$, then $\mathbb{S}^3\setminus\CS(W)$ is a manifold and we can use the Steenrod-Sitnikov duality to get:
    \begin{align*}
        \mathbb{Z}^n\oplus(\mathbb{Z}_2)^n & \cong H_1(\CS(Y))\cong {}^{st}\hspace{-0.2em}H_1(\CS(W)) \\
        &\cong \check{H}^1(\mathbb{S}^3\setminus \CS(W)) \cong H^1(\mathbb{S}^3\setminus\CS(W))         
    \end{align*}
    This is a contradiction since the first cohomology group must be torsion-free. 
\end{proof}

\begin{theorem}
    Every plane continuum that separates the plane into a finite number of components inscribes rectangles.
\end{theorem}

\begin{proof}
    The proof can be carried out in the same way as the proof of Theorem~\ref{2compos}, with the only difference that we now reach a contradiction with Proposition~\ref{CSNoEmbNcompos} (instead of with Proposition~\ref{CSNoEmb2compos}). 
\end{proof}

\subsection{Continua that separate the plane into an infinite number of components}

In this section we will prove that a plane continuum that separates the plane into an infinite number of components inscribes rectangles.

\begin{lemma}\label{HomCSHawaii}
Let $\mathcal{H} = \varprojlim Y_k$ denote the Hawaiian earring, where $Y_k = \vee_{i=1}^k \Si$. The first Steenrod homology group of the 2-symmetric clamshell of $\CS(\mathcal{H})$ satisfies:
$$ \Ste(\CS(\mathcal{H})) \cong \mathbb{Z}^\omega \oplus (\mathbb{Z}_2)^\omega$$
\end{lemma}

\begin{proof}
    
    As in the proof of Lemma~\ref{HomCSnWedge}, we can consider the reduced Steenrod Mayer-Vietoris exact sequence for the decomposition $(\CS(\mathcal{H}), \CS(\mathcal{H})^{\pplus},\CS(\mathcal{H}^{\mminus}))$:

     \begin{align*}
        \ldots\to
         \Ste(\mathcal{H}^\star)&\xrightarrow{\varphi} 
         \Ste(\CS(\mathcal{H})^{\pplus})\oplus \Ste(\CS(\mathcal{H})^{\mminus})\\ \notag
         &\xrightarrow{\psi}
         \Ste(\CS(\mathcal{H}))\to 
         \Stek_0(\mathcal{H}^\star)\to\ldots,
    \end{align*}
    where $\varphi=(i^{+}_*,-i^{-}_*)$ and $i^{+}_*, i^{-}_*$ are the homomorphism induced by the inclusion maps $i^{\ppoint}:\mathcal{H}^\star\to\CS(\mathcal{H})^{\ppoint}.$
    
    On the one hand, since we are using the reduced version of the Steenrod homology and $\mathcal{H}$ is path connected, we have $\Stek_0(\mathcal{H}^\star)=0$. 
    On the other hand, since $\CS(\mathcal{H}^{\ppoint})\simeq F_2(\mathcal{H})$, from~\ref{HawaiiHomolo} we get $ \Ste(\CS(\mathcal{H})^{\ppoint})\cong \Zset^\omega$. 
    Hence, the previous sequence turns into: 
    \begin{equation}\label{MVH}
        \Zset^\omega \xrightarrow{\varphi} \Zset^\omega\oplus\Zset^\omega\xrightarrow{\psi}\Ste(\CS(\mathcal{H}))\to 0.
    \end{equation}
    Observe that $\psi$ is then a surjection and that, as in Lemma~\ref{HomCSnWedge}, each copy of $\Si$ in $\mathcal{H}$
    is mapped, under the inclusion maps $i^{\ppoint}$, to the boundary of a M\"obius strip that winds twice around the core of the strip. Therefore, $i^{\ppoint}$ maps each generator to its double, so $\text{Im}\varphi=\langle (\bar{2},-\bar{2}) \rangle$, where $\bar{a}$ denotes the constant sequence $(a,a\ldots)$. These observations, together with~\eqref{MVH}, yield to:
    $$\Ste(\CS(\mathcal{H}))\cong \frac{\Zset^\omega\oplus\Zset^\omega}{\langle (\bar{2},-\bar{2}) \rangle}\cong \Zset^\omega\oplus (\Zset_2)^\omega.$$
    \end{proof}

\begin{proposition}\label{CSNoEmbOmCompos}
    Let $W$ be a plane separating continuum, such that $\mathbb{R}^2 \setminus W$ has infinite components, then $\CS(W)$ is not embeddable in $\mathbb{R}^3$.
\end{proposition}

\begin{proof}
    Let $W$ as in the statement, from Lemma~\ref{sameshape} and Remark~\ref{rem:onhomo}, it follows that the Steenrod homology groups of $\CS(W)$ are isomorphic to the Steenrod homology groups of the Hawaiian earing $\mathcal{H}$, thus $\Ste(\CS(W))\cong \Zset^\omega\oplus(\Zset_2)^\omega$.

    Assume that there is an embedding of $\CS(W)$ into $\Rset^3$, then $\mathbb{S}^3\setminus \CS(W)$ is a manifold and from the Steenrod-Sitnikov duality, we get:
    \begin{align*}
        \mathbb{Z}^\omega\oplus(\mathbb{Z}_2)^\omega & \cong \Ste(\CS(\mathcal{H}))\cong \Ste(\CS(W)) \\
        &\cong \check{H}^1(\mathbb{S}^3\setminus \CS(W)) \cong H^1(\mathbb{S}^3\setminus\CS(W))         
    \end{align*}
    This is a contradiction since the first cohomology group must be torsion free.     
\end{proof}  

\begin{theorem} 
    Every plane continuum that separates the plane into an infinite number of components inscribes rectangles.
\end{theorem} 

\begin{proof}
    Let $W$ be a continuum as in the statement. 
    As in the proof of Theorem~\ref{2compos}, if we assume that $W$
    does not inscribe rectangles, then the Vaughan's function (see Remark ~\ref{Vau}) defines an embedding of $\CS(W)$ into $\Rset^3$, which in turn contradicts Proposition~\ref{CSNoEmbOmCompos}. 
\end{proof}

\begin{Coro}
    The Wada continua for $\omega$ lakes inscribe rectangles.
    Every non-unicoherent plane continuum inscribes rectangles.
\end{Coro}

\section{Conclusions and further questions}

From the results of the previous section we readily derive:

\begin{theonn}
    Every plane separating continuum inscribes rectangles.
\end{theonn}

 We would like to point out here that plane separating continua that do not contain a Jordan curve are ubiquitous in many instances: in the space of all plane separating continua with the Hausdorff metric these continua contain a dense $G_\delta$ set, see \cite[Theorem 11]{BingECL}.
 On the other hand, the pseudo-circle and the Wada continua appear in abundance as fibers of continuous maps from $\mathbb{S}^2$ to the unit interval, see ~\cite{darji}. In dynamical systems, the pseudo-circle emerges as a minimal set of an area preserving, $C^\infty$-diffeomorphism~\cite{handel}, and it is a generic inverse limit of Lebesgue measure preserving maps of degree 1 from $\Si$ to itself~\cite{jernej}. 

We finish this work by stating some related questions that we find interesting and whose solution might help to better understand the relation between the inscription properties of a plane topological space, its local behavior and its ``shape'' in the more broader sense.

Evidently, it would be interesting to know what can be said about the square peg problem and the rectangular peg problem versions for plane separating continua:
\begin{enumerate}

    \item Does every plane separating continuum inscribe squares?

    \item Given a rectangle, does every plane separating continuum inscribe a rectangle similar to the given one? 
\end{enumerate}

On the other hand, there are interesting questions for non-separating plane continua. The pseudo-arc shares distinctive properties with both the circle and the arc. We find interesting the problem of finding an embedding of the pseudo-arc into the plane that does not admit an inscribed rectangle. These are related questions:

\begin{enumerate}
    \setcounter{enumi}{2}
    \item Does every non degenerate homogeneous plane continua inscribe rectangles? by the classification of homogeneous plane continua this is the same as asking if the pseudo-arc inscribes rectangles.
    
    \item Is there a non-separating indecomposable continuum that does not inscribe rectangles? e.g. does the Knaster continuum inscribe rectangles? 

    \item Is there a plane separating continuum that admits only a countable or even finite number of inscribed rectangles?

    \item Let $X$ be a plane separating continuum of topological dimension $1$. Given a polygon $P$ with $n\geq 5$ vertices, can we always find an embedding $\gamma:X\to\mathbb{R}^2$ such that $\gamma(X)$ does not admit an inscribed polygon $P$?
\end{enumerate}

\hspace{1pt}

\bibliographystyle{amsplain}
\bibliography{biblio}

\end{document}